\documentclass[11pt]{article}
\usepackage[left=1in,right=1in,bottom=1in,top=1in]{geometry}
\usepackage{amsmath}
\usepackage{amsfonts}
\usepackage{amssymb}
\usepackage{amsthm}
\usepackage{young}
\usepackage[vcentermath]{youngtab}
\usepackage{tikz}

\usepackage{algorithm}
\usepackage[noend]{algpseudocode}
\usepackage{amsmath, amssymb, amsthm}
\usepackage{blindtext}
\usepackage{hyperref} 
\usepackage{cleveref}
\usepackage{comment}
\usepackage{enumerate}
\usepackage{fullpage}
\usepackage{mathtools}
\usepackage{subcaption}

\makeatletter
\newtheorem*{rep@theorem}{\rep@title}
\newcommand{\newreptheorem}[2]{%
\newenvironment{rep#1}[1]{%
 \def\rep@title{#2 \ref{##1}}%
 \begin{rep@theorem}}%
 {\end{rep@theorem}}}
\makeatother

\parskip \medskipamount
\newtheorem{theorem}{Theorem}
\newtheorem{lemma}[theorem]{Lemma}
\newreptheorem{theorem}{Theorem}

\newtheorem{remark}[theorem]{Remark}
\newtheorem{definition}[theorem]{Definition}
\newtheorem{proposition}[theorem]{Proposition}

\newcommand{\pr}[1]{\mathbb{P}\!\left(#1\right)}
\newcommand{\expect}[1]{\mathbb{E}\!\left[#1\right]}
\newcommand{\var}[1]{\operatorname{Var}\!\left[#1\right]}
\newcommand{\prcond}[2]{\mathbb{P}\!\left(#1\;\middle\vert\;#2\right)}

\title{Optimal strong stationary times for random walks on the chambers of a hyperplane arrangement}

\date{}

\author{Evita Nestoridi \thanks{Princeton University, Department of Mathematics.
 Email: \href{exn@princeton.edu}{\texttt{exn@princeton.edu}}}}

\begin{document}
\maketitle
\begin{abstract}
This paper studies Markov chains on the chambers of real hyperplane arrangements, a model that generalizes famous examples, such as the Tsetlin library and riffle shuffles. We discuss cutoff for the Tsetlin library for general weights, and we give an exact formula for the separation distance for the hyperplane arrangement walk. We introduce lower bounds, which allow for the first time to study cutoff for hyperplane arrangement walks under certain conditions. Using similar techniques, we also prove a uniform lower bound for the mixing time of Glauber dynamics on a monotone system.
\end{abstract}

\section{Introduction}
Hyperplane arrangement walks have been studied by many different authors. The eigenvalues and eigenvectors for such walks have been characterized (\cite{BHR}, \cite{BD}, \cite{AthD}, \cite{Pike}), while a coupling argument \cite{AthD} gives upper bounds on the total variation distance that are not always sharp and there are no lower bounds. We introduce  a strong stationary time, which can be used to give both upper and lower bounds for the separation distance mixing time. In many cases, this strong stationary time is optimal and leads to a proof of existence of cutoff. This is the first time that lower bounds and cutoff are discussed in the setting of hyperplane arrangement walks.
 
%
In this paper, we denote by $\mathcal{A}$ a central  hyperplane arrangement, that is a finite collection of hyperplanes in $\mathbb{R}^n$ that pass through the origin. 
The hyperplanes of $\mathcal{A}$ cut $\mathbb{R}^n$ in finitely many connected, open components, called chambers. A chamber can be expressed as a vector with $m=|\mathcal{A}|$ coordinates,  by keeping track of whether it is on the ($+$) or ($-$) half space of each hyperplane. Similarly we can express every face as a vector with $m$ coordinates, where each coordinate is either $+$, $-$ or $0$, where $0$ indicates that the face lies on the hyperplane. Denote by $\mathcal{C}$ be the set of chambers and $\mathcal{F}$ denote the set of faces. 

Let $w$ be a probability measure on $\mathcal{F}$ such that for every hyperplane $H \in \mathcal{A}$ there is an $F$ with $w(F)>0$ and $F \nsubseteq H$. Given such a $w$, we study the following process on $\mathcal{C}$: from $C \in \mathcal{C}$, choose $F$ according to $w$ and move to $FC$, that is the unique chamber neighboring $F$ which is the nearest to $C$ (in the sense of crossing the smallest number of hyperplanes). Algebraically, $FC$ is the chamber whose coordinates agree with the non-zero coordinates of $F$, while the rest are determined by $C$. This action $\mathcal{F} \times \mathcal{C} \rightarrow \mathcal{C}$ extends naturally to an associative product $\mathcal{F} \times \mathcal{F} \rightarrow \mathcal{F}$. Let $T$ be the first time that the product of the faces picked is a chamber. The main theorem of this paper provides tools for obtaining upper and lower bounds for the mixing time and studying cutoff for this Markov chain.

\begin{theorem}\label{opt2}
 Assume that a group $G$ acts on $\mathbb{R}^n$ preserving the hyperplane arrangement $\mathcal{A}$ and acting transitively on the set of chambers. If the weights $w$ are $G$-invariant then  $$s(t)= \pr{T>t}$$ for every $t$, where $s(t)$ is the separation distance of the walk at time $t$.
\end{theorem}

Theorem $3.3$ of \cite{FLM} discusses lower bounds and upper bounds for $\pr{T>t}$ and it gives cutoff for cases such as when all faces with positive weight have at most $M$ non-zero coordinates, where $M=m^{o(1)}$, and $b = \sum_{F \nsubseteq H_i }w(F)$ is independent of $i$.

To discuss cases where a face with large number of non-zero coordinates might be collected, just as in the case of riffle shuffles, assume that there are $b$ and $d$ bounded away from $1$ so that for every $i,j$ we have that
 \begin{equation}\label{coupn}
b = \sum_{F \nsubseteq H_i }w(F)
\mbox{ and } 
 \sum_{F \nsubseteq H_i \cup H_j}w(F)= d.
\end{equation}
If $d=0$ then discussing cutoff for the hyperplane arrangement walk under the assumptions of Theorem \ref{opt2} is merely a coupon collecting problem. The following proposition discusses cutoff for the cases that Theorem \ref{opt2} applies to.
\begin{proposition}\label{cut}
Let $\mathcal{A}$ be a central hyperplane arrangement that satisfies the assumptions of Theorem \ref{opt2} and equation \eqref{coupn} with $b$ bounded away from one,  $b \leq \frac{1+d}{2}$ and $0 < d \leq b^2$.
Then the random walk on the chambers of $\mathcal{A}$ exhibits cutoff with respect to the separation distance at time $ \log_{\frac{1}{1-b}} m$ with window $\frac{1}{b}$.
\end{proposition}

Equation \eqref{coupn} holds if for example the $G$-action is $2$-transitive on the hyperplanes. The assumptions of Proposition \ref{cut} hold for the riffle shuffles, which is introduced in Section \ref{braid}, since $b= \frac{1}{2}$ and $d=\frac{1}{4}$. It can also be applied for the case of a non-local random walk on the hypercube as explained on Section \ref{examples}.

To give explicit bounds for the separation distance mixing time in the general case or to say whether there is cutoff, one needs to study a complicated coupon collector problem. Coupon collecting with different weights has been studied in special cases by Diaconis \cite{cutoff}, Fill \cite{Fill-Tsetlin} and Stadje \cite{Stadje}. 

We study specifically the Tsetlin library: pick a book $i$ with probability $w_i$ and move it to the front. This is a very well studied Markov chain mainly because of its use in dynamic file maintenance and cache maintenance (\cite{Do}, \cite{FHo}, \cite{Phatarfod}). Fill \cite{Fill-Tsetlin} has already studied the separation distance for this model and has proven cutoff for special cases of weights. We use poissonization to prove the following new cutoff result:
\begin{theorem}\label{Tse}
For the Tsetlin Library with weights $w_i$, let $t_*$ be the unique solution to $\sum_{i=1}^n e^{-w_it_*}= \frac{1}{2}$. If $t_* \min_{i} \{ w_i\} \rightarrow \infty$ and $t_* \min_{i} \{ w^2_i\}$ is bounded for all $n$, then
\begin{enumerate}
\item[(a)] If $t= t_* + c \frac{1}{\min_{i} \{ w_i\} }$, then 
$$s(t) \leq 1- e^{-\frac{e^{-\frac{c}{2}}}{2}} +\frac{  4\min_{i} \{ w^2_i\}t_* + 2c \min_{i} \{ w_i\} }{c^2},$$
where $c>0$.
\item[(b)]  If $t= t_* -  2c\frac{1}{\min_{i} \{ w_i\} }$, then 
$$s(t) \geq     1- e^{-\frac{e^{\frac{c}{2}}}{2}}- \frac{4 \min \{w_i^2\}t_* -2c \min \{w_i\} }{c^2}  -  \frac{1}{c},$$
where $0<c < \frac{t_*\min_{i} \{ w_i\}}{  2}$.
\end{enumerate}
where $s(t)$ is the separation distance.
\end{theorem}
In section \ref{examples}, we prove a similar result for the nearest random walk on the hypercube with general weights.
The following theorem discusses a general lower bound for the separation distance.
\begin{theorem}\label{hyperplane}
For every central hyperplane arrangement, we have that
$$\pr{T>t} \leq s(t) ,$$
where $\pi(D) $ is the stationary measure and $s(t)$ is the separation distance. 
\end{theorem}

The technique used in the proof of Theorem \ref{hyperplane} can be used in many other applications of broader interest. For example, section \ref{Ising} gives a uniform lower bound for the mixing time of  Glauber dynamics on monotone systems. A system $(\Omega, S,V, \pi)$ consists of a finite set $S$ of spins, a $V$ set of sites, the set of configurations $\Omega := S^V$ and a positive probability distribution on $\Omega$ which will serve as the stationary measure for the Glauber dynamics. If $S$ is totally ordered then $S^V$ is endowed with the coordinate-wise partial order. Following the notation of Peres and Winkler  \cite{PeresWinkler}, let $\sigma^s_u$ denote the configuration obtained from $\sigma$ by updating its value at $u$ to $s$. Let $\sigma^{\bullet}_u$ denote the set of possible configurations occurring by updating the value of $\sigma$ at $u$. A system $(\Omega, S,V, \pi)$ is called monotone if whenever $\sigma \leq \tau$ are two configurations in $\Omega$, then
$$\bigg\{ \frac{\pi(\sigma^s_u)}{\pi(\sigma^{\bullet}_u)} \bigg\}_{s \in S}  \preceq \bigg\{ \frac{\pi(\tau^s_u)}{\pi(\tau^{\bullet}_u)}\bigg\}_{s \in S}, $$
where $\preceq$ indicates that $\bigg\{ \frac{\pi(\tau^s_u)}{\pi(\tau^{\bullet}_u)}\bigg\}_{s \in S} $ stochastically dominates $\bigg\{ \frac{\pi(\sigma^s_u)}{\pi(\sigma^{\bullet}_u)} \bigg\}_{s \in S} $.

We introduce a short proof that generalizes the uniform lower bound for the separation distance and total variation mixing times proven by Ding and Peres in the special case of the Ising model \cite{DingPeres}.
\begin{theorem}\label{mon}
Let $(\Omega, S,V, \pi)$ be a monotone system, with $|V|=n$ number of sites and $\Omega \subset S^V$. Then for the Glauber dynamics on $\Omega$, if $t= n \log n-cn$, then 
$$s(t)\geq 1- e^{-e^c}$$
and if $t=\frac{1}{2}n \log n -cn$, then 
$$d(t) \geq \frac{1}{4} -\frac{1}{4} e^{-e^c},$$
where $c>0$ and $d(t)$ is the total variation distance at time $t$.
\end{theorem}
\begin{remark}
In the appendix of \cite{DingPeresarxiv}, Ding and Peres provided an alternative proof of the uniform lower bound for the Ising model that could be generalized in the case of monotone systems.
\end{remark}

\subsection{Literature and organization of the paper}\label{his}
Bidigare, Hanlon and Rockmore (BHR) \cite{BHR} defined the walk on $\mathcal{C}$ and characterized its eigenvalues. Brown and Diaconis \cite{BD} proved that the transition matrix of this Markov chain is diagonalizable and they reproved the BHR result. They also found a necessary and sufficient condition on $w$ so that the walk has a unique stationary distribution. This condition is that $w$ separates the hyperplanes of $\mathcal{A}$, namely that for every $H \in \mathcal{A}$ there is a face $F \nsubseteq H$ such that $w(F)>0$. Under this assumption, they provide a stochastic description for the stationary measure $\pi$: sample without replacement from $w$ and apply these faces in reverse order to any starting chamber (this way the first chosen face is the last to be applied). 
In this paper, $w$ is assumed to be separating, so there exists a notion of convergence to this unique distribution. Athanasiadis and Diaconis have a similar discussion in \cite{AthD}, where they use purely combinatorial methods as well as a coupling argument. Their upper bounds are introduced for the total variation distance. When applied to examples though, these upper bounds are not sharp. Then Brown \cite{Brown} generalized this coupling for the case of semigroups. Pike has studied the eigenvectors of the process in \cite{Pike}.

Some important examples can be found in sections \ref{braid} and \ref{examples}. Section \ref{thm} presents the proof of Theorem \ref{hyperplane}. Finally, section \ref{last} contains the proofs of Theorems \ref{opt2} and Proposition \ref{cut}. Section \ref{Ising} discusses how the techniques of this paper can be used in monotone systems, such as the Ising model and presents the proof of Theorem \ref{mon}.

 \section{Preliminaries}
\subsection{Strong stationary times and cutoff}\label{sep}
Denote by $C^t_{x_0}$ the $t$th configuration of the walk that starts at $x_0$, that is the chamber the walk is on after $t$ steps of running the process.  
Define the separation distance to be
$$s(t)= \max_{x_0 \in \mathcal{C}} \left( 1-  \min_{x \in \mathcal{C}} \frac{\pr{C^t_{x_0}=x}}{\pi(x)} \right).$$
To bound the separation distance of a Markov chain, Diaconis and Aldous \cite{A-P} introduced the following definition.
\begin{definition}
Fix $x_0\in \mathcal{C}$. A strong stationary time is a stopping time $\tau$ such that for every $A \subset \mathcal{C}$ and $t\geq 0$ it holds that
$$\mathbb{P}\left(C_{x_0}^t \in A \vert \tau \leq t\right)= \pi(A),$$
where $C^t$ is the state that the Markov Chain is at time $t$ and $\pi$ is the stationary measure.
\end{definition}
Aldous and Diaconis \cite{A-P} proved the following theorem which is the main link between strong stationary times and separation distance:

\begin{lemma}\label{inequality}
If $\tau$ is a strong stationary time then for $t>0$,
$$s(t) \leq \pr{\tau > t}.$$
\end{lemma}
This paper studies cutoff with respect to the separation distance. The sequence of walks on the chambers of hyperplane arrangement $\mathcal{A}_n$ will be said to exhibit cutoff at $t_n$ with window $w_n=o(t_n)$ if and only if
$$\lim_{c \rightarrow \infty} \lim_{n \rightarrow \infty} s(t_n-cw_n)= 1 \mbox{ and } \lim_{c \rightarrow \infty} \lim_{n \rightarrow \infty} s(t_n+cw_n)= 0,$$
where $n$ is the dimension of $\mathbb{R}^n$.

\subsection{The rigorous definition of the product}\label{hyp}
This section presents the rigorous setup of the hyperplane arrangement theory. It defines the chambers, the faces and the product algebraically. 

Any face $F$ can be written in the following form:
$$F= \bigcap_{i \in I} H^{\sigma_i(F)}_i,$$
where $\sigma_i(F) \in \{+,-,0 \}$, $H^+_i$ corresponds to the right open half-space determined by $H_i$ (and respectively $H^-_i$ for the left one) and $H^0_i= H_i$. Notice that  $\sigma_i(F)\neq 0$ for all $i$ if and only if $F$ is a chamber. The faces form a semigroup under the following associative product:
\begin{definition}
If $F,G$ are two faces then
$$FG= \bigcap_{i \in I} H^{\sigma_i(FG)}_i,$$
where 
$$\sigma_i(FG)=\begin{cases} \sigma_i(F), & \mbox{if } \sigma_i(F)\neq 0 \\ 
 \\ \sigma_i(G), &  \mbox{ otherwise.} \end{cases}$$
\end{definition}
Notice that $C_{x_0}^t=F^t \ldots F^1x_0$, where $F^t$ if the face picked at time $t$.
Brown \cite{Brown} explains that multiplication of faces satisfies both the ``idempotence" and the ``deletion property", that is if $F$ and $G$ are two faces then
\begin{equation}\label{deletion}
F  F=F \mbox{ and } FGF=FG
\end{equation}

\section{Braid Arrangement and card shuffles}\label{braid}
Many shuffling schemes can be viewed as Markov chains on the chambers of the braid arrangement.
 The braid arrangement consists of the hyperplanes 
\begin{equation}
x_i=x_j.
\end{equation}
%
%
%
It is clear that the chambers are in one to one correspondence with the elements of $S_n$. The faces are exactly the ordered block partitions of $[n]$. For example,
\begin{equation*}
\{1,2,3\}\{4,5\} \{6,7,\ldots n\}
\end{equation*} 
corresponds to 
\begin{align*}
 x_1=x_2=x_3< x_4=x_5< x_6=x_7=\ldots = x_n 
\end{align*}

One general card shuffling scheme studied in this section is the "pop shuffles": consider all ordered block partitions of $[n]= \{1,2,\ldots n \}$ and assign weights to them. In this shuffling scheme one picks an ordered block partition $A_1, A_2,\ldots A_m$ according to the weights and remove from the deck the cards indicated by $A_1$ and put them on the top, keeping their relative order fixed. Then remove the cards indicated by $A_2$ and put them exactly below the $A_1$ cards, keeping their relative order fixed and, so on. This is exactly the hyperplane arrangement walk.

Proposition \ref{cut} can be applied in this case, since $G=S_n$ acts transitively on the chambers. Assuming that the weights satisfy \eqref{coupn}, the following walks exhibit cutoff.
\paragraph{Studying cutoff for the Tsetlin Library.}\label{wrtt}
Let $w_j$ denote the positive weight assigned to the $j$th card. Consider the following Markov Chain on $S_n$: start from a state $x$ in $S_n$. With probability $w(j)$ remove card $j$ and place it on top.
The stationary distribution is the Luce model, which is described as sampling from an urn with $n$ balls without replacement, picking ball $j$ with probability $w(j)$.

The eigenvalues of this process were discovered independently by Donnelly \cite{Do}, Kapoor and Reingold \cite{KR}, and Phatarfod \cite{Phatarfod}.  Brown and Diaconis \cite{BD}, Athanasiadis and Diaconis \cite{AthD} also present the eigenvalues of the Tsetlin Library as an example of a hyperplane walk. Fill gives an exact formula for the probability of any permutation after any number of moves in \cite{Fill-Tsetlin} and discusses cutoff for specific types of weights (see Theorem $4.3$ of \cite{Fill-Tsetlin}). 

\begin{lemma}[Fill, \cite{Fill-Tsetlin}]
For the Tsetlin Library with weights $w(i)$, let $T$ be the first time we have touched $n-1$ cards. Then $$s(t)= \pr{T>t}.$$
\end{lemma}
Theorem \ref{Tse} is one of the main results of this paper that is not a special case of Theorem \ref{opt2}.
\begin{proof}[Proof of Theorem \ref{Tse}]
Let $T_n$ be the first time that we have picked all cards. 
\begin{enumerate}
\item[(a)] Let $X_i(t)$ be the number of times that card $i$ has been picked at time $t$. Let  $\overline{t}= t_* + \frac{c}{2\min_{i} \{ w_i\}} $ and $Y_i$ be independent random variables following $Pois(w_i \overline{t}) $. So $Y=\sum_i Y_i \thicksim Pois(\overline{t})$. Now $(X_1(t), \ldots, X_n(t))$ at time $t$ have the same joint distribution as $(Y_1, \ldots, Y_n) \vert \sum_i Y_i= t$ namely they both follow a multinomial distribution:
$$\prcond{(Y_1, \ldots, Y_n)=(a_1,\ldots, a_n)}{ \sum_i Y_i= t}= \frac{\frac{e^{-\overline{t}} \prod (w_i\overline{t})^{a_i}}{\prod a_i!}}{\frac{e^{-\overline{t}} \overline{t}^t}{t!}}= t! \prod_i  \frac{w_i^{a_i} }{a_i!}$$$$= \pr{(X_1(t), \ldots, X_n(t))=(a_1, \ldots, a_n)}$$
for every non-negative integers $a_1,\ldots a_n$ whose sum is $t$. Then 
\begin{align}
& s(t)=\pr{T> t } \leq \pr{T_n> t }=\cr
&1- \pr {X_i\left( t \right) \geq 1, \forall i}= 1- \prcond{Y_i \geq 1, \forall i}{\sum_i Y_i=t} \cr
 &\label{increasing} \leq 1 - \sum_{k \leq t} \prcond{ Y_i \geq 1, \forall i}{\sum_i Y_i=k} \pr{\sum_i Y_i=k}\\
&=  1- \pr { \forall i \mbox{ }Y_i \geq 1 , \sum_i Y_i\leq t }\cr
&  \label{finish} \leq 1- \pr{Y_i\geq 1, \forall i} +\pr{ \sum_i Y_i> t},
\end{align}
where \eqref{increasing} holds because $\pr {Y_i \geq 1, \forall i \vert \sum_i Y_i=k }$ is increasing on $k$. Since $\overline{t}= t_* + \frac{c}{2\min_{i} \{ w_i\}} $ and $t= t_* + \frac{c}{\min_{i} \{ w_i\}} $, we have that
\begin{align}
 \eqref{finish} & = 1- \prod (1- e^{-w_i \overline{t}} )  + \pr{  \frac{\sum_iY_i -\overline{t}}{\sqrt{\overline{t}}} > \frac{ \frac{c}{2\min_{i} \{ w_i\}}}{\sqrt{t_* + \frac{c}{2\min_{i} \{ w_i\}}}} } \cr
& \lesssim  \label{Cheb} 1- e^{-\sum_{i=1}^n e^{-w_i\left(t_* + \frac{c}{2\min_{i} \{ w_i\}} \right)}}  +\frac{  4\min_{i} \{ w^2_i\}t_* + 2c \min_{i} \{ w_i\} }{c^2}\\
& \label{explain} \leq 1- e^{-\frac{e^{-\frac{c}{2}}}{2}} +  o(1),
\end{align}
where \eqref{Cheb} is because of Chebychev's inequality and in \ref{explain}, we used the fact that $\sum_{i=1}^n e^{-w_it_*}= \frac{1}{2}$.
\item[(b)]
To prove the lower bound, we use again that $\pr{T>t} $ behaves similarly to $\pr{T_n>t} $. This is illustrated by the following inequality: 
$$ \pr{T> t} \geq \pr{T_n> t  +\frac{c}{\min_{i} \{ w_i\} } }  - \pr{T_n-T \geq  \frac{c}{\min_{i} \{ w_i\} }}.$$
Let $\overline{t}= t_* -  \frac{c}{2\min_{i} \{ w_i\} }$ and let $Y_i$ be a random variable following $Pois(w_i \overline{t}) $. Then again $Y=\sum_i Y_i \thicksim Pois( \overline{t})$ and using Markov's inequality, we have that
\begin{align*}
& \pr{T> t} \geq \pr{T_n> t  +\frac{c}{\min_{i} \{ w_i\} } }  - \pr{T_n-T \geq  \frac{c}{\min_{i} \{ w_i\} }}  \\
& \geq1- \pr {X_i\left( t  +\frac{c}{\min_{i} \{ w_i\} } \right) \geq 1, \forall i} - \frac{1}{c} \\
& =1- \prcond{Y_i \geq 1, \forall i}{\sum_i Y_i=t  +\frac{c}{\min_{i} \{ w_i\} }}  - \frac{1}{c}\\
&= 1-\sum_{k=1}^{\infty}\prcond{Y_i \geq 1, \forall i}{\sum_i Y_i=t  +\frac{c}{\min_{i} \{ w_i\} }} \pr{\sum_i Y_i=k}- \frac{1}{c} \\
&\geq 1- \pr{Y_i \geq 1, \forall i,\sum_i Y_i\geq t  +\frac{c}{\min_{i} \{ w_i\} }} - \pr{\sum_i Y_i < t +\frac{c}{\min_{i} \{ w_i\} }}  - \frac{1}{c}  
\end{align*}
\begin{align*}
& > 1- \pr{Y_i \geq 1, \forall i } - \pr{\frac{\sum_iY_i -\overline{t}}{\sqrt{\overline{t}}} <- \frac{ \frac{c}{2\min_{i} \{ w_i\}}}{\sqrt{t_* - \frac{c}{2\min_{i} \{ w_i\}}}} }  - \frac{1}{c} \\
&\gtrsim  1- e^{-\frac{e^{\frac{c}{2}}}{2}}- \frac{4t_* \min \{w_i^2\} -2c \min \{w_i\} }{c^2}  -  \frac{1}{c}
\end{align*}
\end{enumerate}
\end{proof}
\paragraph{Inverse Riffle Shuffles.}
In inverse riffle shuffles, as presented by Aldous and Diaconis \cite{A-P}, one marks each card with zero or one with probability $1/2$, then moves the ones marked with zero on top, preserving their relative order. This corresponds to sampling among the two-block ordered partitions $\{c_1,c_2,\ldots, c_i\} \{[n] \setminus \{c_1,c_2,\ldots, c_i\} \}$.

Although Bayer and Diaconis \cite{BaD} prove that the optimal upper bound for the total variation mixing time is $\frac{3}{2} \log_2 n$, yet work done by Aldous and Diaconis \cite{A-P} and Assaf, Diaconis and Soundararajan \cite{ADS} proves that the separation distance mixing time is $2 \log_2 n $. Several other metrics have also been studied: \cite{ADS} have studied the $l^{\infty}$ norm as well, while Stark, Gannesh and O'Connell \cite{SGO} studied the Kullback-Leibler distance.

 As discussed in Athanasiadis and Diaconis in \cite{AthD} there is a generalization of this card shuffling, namely marking the cards with a number in $\{0,1,\ldots,a-1\}$ according to the multinomial distribution. Then move the ones marked with zeros on top, keeping their relative order fixed, and continue with the ones marked with $1$ etc. This is a generalization of a strong stationary argument of Aldous' and Diaconis' in \cite{A-P}, giving an upper bound for the general inverse riffle shuffle of the form $2 \log_a n$, while Proposition \ref{cut} guarantees the existence of cutoff as explained thoroughly in Section \ref{last}.
Following Proposition \ref{cut}, we prove that the generalized riffle shuffles has cutoff for the separation distance at $2 \log_a n$ with constant window. 

\paragraph{$k$ to top.}
A natural generalization of random to top is to pick $k$ cards at random and move them to the top, keeping their relative order fixed. This random walk is generated by the uniform measure on faces of the form  
$$\{c_1,c_2,\ldots, c_k\} \{[n] \setminus \{c_1,c_2,\ldots, c_k\} \}$$
This is a new example of card shuffling that falls in the category of hyperplane arrangement walks. Notice that the mixing time for $k$ and $n-k$ are the same. If $1<k \leq \frac{n}{2}$ then $b= \frac{k}{n}$ and $d= \frac{k^2}{n^2}- \frac{k(n-k)}{n^2(n-1)}$ satisfy the conditions of Proposition \ref{cut} and, therefore, the walk exhibits cutoff with respect to the separation distance at $2\log_{\frac{n}{n-k}} n$ with window $\frac{n}{k}$.

\paragraph{Random to top or bottom.}
Consider the card shuffling where a card is chosen at random and is moved to the top or the bottom with probability $1/2$. This is again a random walk on the chambers of the braid arrangement. The faces used are of the form $\{\{c\}, \{[n] \setminus \{c\}\}\}$ and $\{ \{[n] \setminus \{c\}\}, \{c\}\}$ each one having weight $1/2n$. Theorem \ref{opt2} says that the first time all cards have been touched is an optimal strong stationary time. Therefore, there is cutoff for the separation distance at $n \log (n)$ with window $n$. Theorem \ref{Tse} holds if we assign weights to the cards. For some calculations see the work of Diaconis \cite{cutoff}.

\section{Boolean Arrangement and hypercube walks}\label{examples}
The Boolean arrangement consists of the hyperplanes $x_i=0$, $1\leq i \leq n$ in $\mathbb{R}^n$. Each chamber is specified by the sign of its coordinates, in other words they are the $2^n$
orthants in $\mathbb{R}^n$.  In other words each chamber corresponds to a vertex of the $n-$dimensional hypercube. The faces are in bijection with $\{-,0,+\}^n$. The projection $FC$ of a chamber $C$ on a face $F$ is a chamber who adopts all the signs non-zero coordinates of $F$ and the rest  of the coordinates have the signs of $C$.
Under the assumptions of Proposition \ref{cut}, the following examples exhibit cutoff since $G= (\mathbb{Z}/2\mathbb{Z})^n$ acts on the chambers transitively. 
\paragraph{The weighted nearest neighbor walk on the hypercube.}
 If the only positively weighted faces are the $e^{\pm}_i$, whose $i$th coordinate is $\pm$ and the rest are zero, then the Markov Chain corresponds to the weighted nearest neighbor random walk on the hypercube, which corresponds to choosing a coordinate and switching it to $\pm$. Denote the weight of $e^{\pm}_i$ by $ w_i^{\pm}$ and assume they are all positive. 
\begin{lemma}
If $w_i^+=w_i^-$, then $T$, the first time that all coordinates have been picked, is an optimal strong stationary time.
\end{lemma} 
\begin{proof}
If $w_i^+=w_i^-$, then the random walk is symmetric and therefore the stationary measure is the uniform measure on $(\mathbb{Z}/2\mathbb{Z})^n$. For $t\geq n$ we have that
\begin{align}
& \label{st} \prcond{C^t_{x_0}=x}{T=t}= \prod_{i=1}^n \frac{w_i^{x(i)}}{w_i^++w_i^-}= \frac{1}{2^n}
\end{align}
for all $x \in (\mathbb{Z}/2\mathbb{Z})^n$, where $x(i)$ denotes the $i$th coordinate of $x$. The first equality in \eqref{st} holds because we only need to keep track of whether we chose $+$ or $-$ the last time the $i$th coordinate was picked.

This means that $T$ is a strong stationary time. Lemma \ref{inequality} and Theorem \ref{hyperplane} say that
$$s(t)=\pr{T>t}$$
which finishes the proof.

\end{proof}
 The following proposition discusses cutoff for the nearest neighbor random walk on the hypercube for the case where $w_i^+=w_i^-$. 
 \begin{proposition}
 Assume that $w_i^+=w_i^-$ and let $w_i= w_i^++w_i^-$. Let $t_*$ be the unique solution to $\sum_{i=1}^n e^{-w_it_*}= \frac{1}{2}$. If $t_* \min_{i} \{ w_i\} \rightarrow \infty$ and $t_* \min_{i} \{ w^2_i\}$ is bounded for all $n$, then the nearest neighbor random walk on the hypercube exhibits cutoff.
 \end{proposition}
 
\begin{proof}
The proof is similar to Theorem \ref{Tse}.
\end{proof}

A special case has been studied by Aldous \cite{Aldous_correct} and Diaconis and Shashahani \cite{PDMS}. They studied the case where $w_i^{\pm}= \frac{1}{2n}$ and they have proved the cutoff for the total variation distance mixing time at $\frac{n}{2} \log n + cn$.

\paragraph{A non-local walk on the hypercube.}
Consider the following walk on the hypercube:  let $1< k \leq \frac{n}{2}$. Pick $k$ coordinates at random and flip a fair coin for each one of them to determine whether to turn them into ones or zeros. Then, $b= \frac{k}{n}$ and $d= \frac{k^2}{n^2}- \frac{k(n-k)}{n^2(n-1)}$ and, therefore, Proposition \ref{cut} says that the walk exhibits cutoff at $ \log_{\frac{n}{n-k}} n$ with window $\frac{n}{k}$.

\section{The proof of Theorem \ref{hyperplane}}\label{thm}

\begin{proof}[Proof of Theorem \ref{hyperplane}]
Fix $t$ and choose a chamber $D $ so that 
$$ \prcond{C_{x_0}^t=D}{T\leq t} \leq \pi(D),$$
Let $x_0$ be the chamber that occurs if we flip all coordinates of $D$. This is possible because all hyperplanes pass through the origin. Let $T$ be the first time that the product of the faces picked is a chamber then for $t \geq a$ we have that
$$C^t_{x_0}(D)= \prcond{C_{x_0}^t=D}{T\leq t} \pr{T\leq t}\leq \pi(D) \pr{T\leq t}.$$
Therefore,
$$s(t)\geq 1-   \frac{\pr{C^t_{x_0}=D}}{\pi(D)} \geq 1-\pr{T\leq t}= \pr{T>t},$$
giving that $s(t) \geq \pr{T>t}$ for all $t>0$, proving Theorem \ref{hyperplane}.
\end{proof}

\section{Cutoff cases-Reflection arrangements}\label{last}
In this section, assume that a group $G$ acts on $\mathbb{R}^n$ preserving the hyperplane arrangement $\mathcal{A}$ so that the action restricted on the chambers is transitive. Assume that the weights are $G-$invariant 
At the end of this section, we prove that under a few more assumptions, the walk exhibits cutoff. As mentioned in page $9$ of Brown and Diaconis \cite{BD} the stationary measure is the uniform measure on the chambers. 
Let $T$ be the first time that the product of the faces picked is a chamber. We can rewrite Theorem \ref{opt2} as 
\begin{lemma}\label{T3}
Let $\mathcal{A}$ be a hyperplane arrangement. If the weights are $G-$invariant then $T$ is a strong stationary time.
\end{lemma}
\begin{proof}
Let $\ell$ be the number of chambers of $\mathcal{A}$. To prove that 
$$\prcond{C_{x_0}^t=C}{T=t}= \pi(C) $$
for every $t$ and $x_0, C \in \mathbb{C}$,  
consider at first the case $t=1$ and remember that $w(c)$ denotes the weight of a chamber $C$ when viewed as a face.
$$\prcond{C^1_{x_0}=C}{T=1}= \frac{w(C)}{\sum_{D \in \mathcal{C}} w(D)} = \frac{1}{\ell}.$$
 Because of the symmetry condition, we have that
$$\prcond{C_{x_0}^2=C}{T\leq 2}=\frac{\sum_{F_{i_1}F_{i_2}=C } w(F_{i_1}) w(F_{i_2})}{\sum_{D \in \mathcal{C} }  \sum_{ F_{i_1}F_{i_2}=D} w(F_{i_1}) w(F_{i_2})}= \frac{1}{\ell} $$
and inductively, because of the weight invariant action of $G$. 
$$\prcond{C_{x_0}^t=C}{T=t} = \frac{1}{\ell}. $$
\end{proof}
\begin{proof}[Proof of Theorem \ref{opt2}]
Theorem \ref{hyperplane} says that 
$$s(t) \geq \pr{T>t} $$
and Lemmas \ref{T3} and \ref{inequality} say that 
$$s(t) \leq \pr{T>t} $$
and therefore we have that 
$$ s(t)=\pr{T>t}. $$
\end{proof}

Having established that $s(t)=\pr{T>t}$ we can now prove Proposition \ref{cut}, which talks about cutoff.
\begin{proof}[Proof of Proposition \ref{cut}]
Let $t=  \log_{\frac{1}{1-b}} m +c \frac{1}{ b}$  then 
$$s(t)=  \pr{T> t}= \pr{\exists H_i\in \mathcal{A}: \forall F \mbox{ picked by time }t, F \subset H_i} \leq m\left( 1-b\right)^t \leq  e^{-c}.$$
So  $\lim_{c \rightarrow \infty} \lim_{n \rightarrow \infty} s(t)= 0$.

To get a lower bound, we need to solve the following coupon collector problem: we are trying $m $ different coupons, where $m $ is the number of hyperplanes in $\mathcal{A}$. Picking a face $F$ with probability $w_F$ corresponds to picking the coupons that correspong to the non-zero coordinates of $F$. Let $X$ be the random variable that describes this coupon collecting in one step. Let $X^t$ denote the set of coupons picked by time $t$. Let $t= \log_{\frac{1}{1-b}} m -c \frac{1}{ b}$, then 
$$\expect{|[m] \setminus X^t|}= m (1-b)^t \approx e^c$$
and 
$$\expect{|[m] \setminus X^t|^2}= \sum_{i,j \notin X^t} \pr{i \notin X^t\mbox{ and }j \notin X^t} = \sum_{i,j \notin X^t} (1-\pr{i \mbox{ or }j \in X })^t= $$ $$ m (1-b)^t  + m(m-1)(1-2b + d)^t.$$
Therefore, the variance is
\begin{align*}
\var{|[m] \setminus X^t|} & = m (1-b)^t + m(m-1)(1-2b + d)^t - m^2(1-b )^{2t} \\
&\leq m (1-b)^t + m(m-1)(1-b)^{2t} - m^2(1-b )^{2t}  \leq e^{c} ,
\end{align*}
because $0 \leq 1-2b + d \leq1-2b + b^2$.
Chebychev's inequality gives that if $t= \frac{1}{b} \log_{\frac{1}{b}} m -c \frac{1}{ b}$ then
$$s(t)=  \pr{T> t}= \pr{|[m] \setminus X^t|\geq 1} \geq 1- \frac{e^c}{(e^c- 1)^2}$$
 and therefore $\lim_{c \rightarrow \infty} \lim_{n \rightarrow \infty} s(t)= 1$.
This finishes the proof that the walk has cutoff at $ \log_{\frac{1}{1-b}} m $ with window $\frac{1}{ b}$.
\end{proof}

\section{A uniform lower bound for monotone systems}\label{Ising}
 This section presents a uniform lower bound for Glauber dynamics in any monotone spin system, as defined in the introduction, generalizing the result of Ding and Peres for the Ising model \cite{DingPeres}. 
 
Let $T$ be the first time we have picked all sites of $V$. Let $x$ be the configuration where all sites of $V$ are assigned with the biggest value of $S$ and let $y$ be the configuration where all sites are assigned with the smallest value of $S$. Let $X^t$ is the configuration of the Glauber dynamics after $t$ steps and let $\pi$ denote the stationary measure. To prove Theorem \ref{mon}, we will need the following lemma:
\begin{lemma}\label{least}
At any time time $t$, we have that 
$$\prcond{X^t=y}{X^0=x, T \leq t} \leq \pi(y).$$
\end{lemma}
\begin{proof}
Lemma 2.1 of Peres and Winkler \cite{PeresWinkler} says that since $y$ is the smallest element in the ordering, $\frac{\prcond{X^t=y}{X^0=x, \mbox{node } v \mbox{ was selected}}}{\pi(y)} $ is the smallest value of $\frac{\prcond{X^t=\cdot}{X^0=x, \mbox{node } v \mbox{ was selected}}}{\pi(\cdot)} $. Inductively, this 
gives that $\frac{\prcond{X^t=y}{X^0=x, T \leq t}}{\pi(y)}$ is the smallest value of  $\frac{\prcond{X^t=\cdot}{X^0=x, T \leq t}}{\pi(\cdot)}$. This value has to be less than or equal to one because both $\prcond{X^t=\cdot}{X^0=x, T \leq t}$ and $\pi$ are probability measures. This finishes the proof.
\end{proof}
\begin{proof}[Proof of Theorem \ref{mon}]
Notice that
$$P^t(x,y)= \prcond{X^t=y}{X^0=x, T \leq t}\pr{T \leq t}.$$
Therefore, $s(t) \geq 1- \frac{\prcond{X^t=y}{X^0=x, T \leq t}\pr{T \leq t}}{\pi(y)}$. Using Lemma \ref{least} we get that 
$$s(t) \geq \pr{T>t}.$$
Just as in  Theorem 1.24 of \cite{BenDoerr}, if $t= n \log n - cn$ then
$$s(t) \geq 1- e^{-e^c}.$$
For the second part, equation (6.12) of \cite{Peresbook}
gives that  if $t= \frac{n}{2} \log n - cn$ then
$$d(t) \geq \frac{1}{4} s(2t ) \geq \frac{1}{4}- \frac{1}{4}e^{-e^c}.$$
\end{proof}

\section{Acknowledgements}
I would like to thank  Persi Diaconis, Bal\' azs Gerencs\' er, Dan Jerison and  Allan Sly for valuable communications.

 \bibliographystyle{plain}
\bibliography{cut-off-hyperplane}

\end{document}